\documentclass
{amsart}
\usepackage{amsmath,amsthm,amsfonts,amssymb,epsfig}
\usepackage{setspace}
\usepackage{stmaryrd}

\usepackage{enumerate}

\usepackage{amsmath,amssymb,amsthm,euro} 
\usepackage{bbm}
\usepackage{stmaryrd}
 \usepackage{tensor,comment,textcomp}
 \usepackage{todonotes}

\newcommand\F{\mathcal F}

\renewcommand{\P}{\mathbb{P}}

\renewcommand{\rho}{\varrho}

\newcommand\eps{\varepsilon}

\newcommand\E{\mathbb{E}}
\newcommand\R{\mathbb{R}}

\newtheorem{theorem}{Theorem}

\newtheorem{corollary}[theorem]{Corollary}

\newtheorem{lemma}[theorem]{Lemma}

\theoremstyle{definition}

\theoremstyle{example}

\newcommand\bF{{\mathbb F}}

\newcommand\bN{{\mathbb N}}
\newcommand\bP{{\mathbb P}}

\newcommand\bR{{\mathbb R}}

\newcommand\cB{\mathcal{B}}

\newcommand\cF{\mathcal{F}}

\newcommand\cL{\mathcal{L}}

\newcommand\cS{\mathcal{S}}

\begin{document}

\title{Applications of pathwise Burkholder-Davis-Gundy inequalities}
\author{Pietro Siorpaes}
\date{}
\maketitle

\begin{abstract}
   We present several applications of the pathwise Burkholder-Davis-Gundy (BDG) inequalities. Most importantly we prove them for cadlag semimartingales and a general function $\Phi$, and use this to derive BDG inequalities (non-pathwise ones) for the Bessel process of order $\alpha\geq 1$ and for martingales stopped at $\tau$, with $\tau$ in a well studied class of \emph{random} times.
   
\bigskip

\end{abstract}
\maketitle

      \section{Introduction}

In recent years a new method of proving martingale inequalities through pathwise counterparts has emerged. This approach, which historically  arose from considerations in robust mathematical finance, has in particular been applied to derive the  pathwise Burkholder-Davis-Gundy (BDG) inequalities: see \cite{Sio13BDG}, where in section 2 one can also find more information on the history of the subject.

The first goal of this paper is to generalize  the  pathwise  BDG inequalities  of  \cite{Sio13BDG}  from  discrete  to continuous time; specifically we to show that, if $X$ is a cadlag semimartingale  and $\Phi$ a very general function, one can explicitly construct  integrands $K,\tilde{K}$ such that for some constant $C_{\Phi}$ the following pathwise BDG inequalities hold
     \begin{align}\label{PBDGintro}
    \Phi(\sqrt{[X]}_t)   \leq C_{\Phi} \Phi(X_t^*) + (K\cdot X)_t \,, \qquad         
     \Phi(X^*_t) \leq C_{\Phi}    \Phi(\sqrt{[X]}_t)   +  (\tilde{K}\cdot X)_t \,  ;
   \end{align} 
    this turns out to be easy if $\Phi(t)=t$ but hard in general, even in the case where $\Phi(t)=t^p$ for some $p> 1$.
    
Trivially the \emph{pathwise} martingale inequalities imply their classical
 equivalent; the second goal of this paper is to show that the pathwise inequalities are strictly more powerful than the classical ones, in useful ways.  
    Indeed, we present several applications of \eqref{PBDGintro}, in which we are able to extend the classical BDG inequalities beyond their traditional domain of validity.
    Although we concentrate our attention exclusively on the BDG inequalities, it is clear that    also for other martingale inequalities the pathwise version is going to be analogously `better' than the classical one; in this regard, it is interesting to keep in mind that every martingale inequality in finite discrete time admits a pathwise equivalent: see \cite{BeNu14, BoNu13}. Let us now review our applications of \eqref{PBDGintro} one by one.
    
    If  $B$ is a $n$-dimensional Brownian motion started at $B_0$, $X:=||B||_{\bR^n}$ and\footnote{We recall that, while in general the BDG inequalities only hold for $p\geq 1$, they hold for any $p>0$ for \emph{continuous} local martingales. The inequalities also hold for very general functions $\Phi$: for the cadlag (resp. continuous) case one can take $\Phi$ as in Theorem \ref{PathBDGCadlag} (resp. Theorem \ref{Bessel}).} $\Phi(t)=t^p$ with $p>0$, the BDG inequalities applied to $B$ imply  that for some $c,C$
\begin{align}
\label{BDGtau}
c\, \E\Phi(\sqrt{[X]_{\tau}})  \leq \E \Phi(X_{\tau}^*)  \leq C \, \E \Phi(\sqrt{[X]_{\tau}})  
\quad \text{ for all  stopping times $\tau$.}
      	 \end{align}
     In other words, the BDG inequalities hold for such a process $X$, even if $X$ is not a local martingale;      $X$ is called an $n$-dimensional Bessel process. 
     More generally (but without making a connection with Brownian motion) one can define    
 the $\alpha$-dimensional Bessel process $X$ for all $\alpha\in \bR$. This is a positive Feller process with continuous paths,   and it is a semimartingale if $\alpha \notin (0,1)$, so   it is natural to ask for which values of $\alpha \notin (0,1)$ the BDG inequalities hold. This questions was answered by  \cite[Theorem 4.1]{GrPe98Be}, where one can find a proof of the fact that  \eqref{BDGtau} holds if $\alpha \geq 1$ (the details being spelled out just in the case $p=1$).
  We will show how this is just a corollary of the pathwise BDG inequalities; so, while the ideas of \cite{GrPe98Be} yield `good' constants\footnote{Meaning constants with the appropriate scaling in $\alpha$.} and ours do not, our approach has the advantage of simplicity. We should say that throughout the paper we make no effort to get good constants; the problem of finding the optimal constants is important and still mostly\footnote{If $\Phi(t)=t^p$ for some values of $p$ the optimal value of $c$ or $C$ is known, see \cite{Os10}.} open.

It turns out that  \eqref{BDGtau} hold not only when $\tau$ is a stopping time, but also for many random times. 
Indeed, if $\tau$ is a finite random time such that $(K\cdot X)_{\tau}$  and $(\tilde{K}\cdot X)_{\tau}$ are in $L^1$ and have zero expectation,  trivially \eqref{PBDGintro} implies that \eqref{BDGtau}.
This can be useful, since given any random time $\tau$ the set of  $M\in H^1:=\{N \text{ is a martingale and } N^*_{\infty}\in L^1(\bP)\}$ for which $\E[M_{\tau}]=0$ is `large' (it has co-dimension $1$ in  $H^1$) and can be quite explicitly characterized: see \cite[Section 3]{Nik07}. Moreover, perhaps surprisingly there are quite a number of interesting examples of random times $\tau$ (called pseudo stopping times) which are not stopping times and for which $\E M_{\tau}=0$ holds for \emph{any} $M\in H^1$; these times have been studied in \cite{NikYor05}, where one can find several equivalent characterizations and examples.

If $\tau$ is a finite\footnote{This is not really needed, as we will see.} stopping time and $A_t:=1_{[\tau,\infty)}(t)$ then  $f(\tau)=\int_0^{\infty} f(s) d A_s $; it is then natural to ask if one can generalize \eqref{BDGtau} and obtain that
 \begin{align} \label{BDGdA}
c\, \E \int_0^{\infty}\Phi(\sqrt{[X]_{s}})  dA_s \leq \E \int_0^{\infty}\Phi(X_{s}^*)  dA_s \leq C \, \E \int_0^{\infty}\Phi(\sqrt{[X]_{s}})   dA_s     
   \end{align} 
holds    for any local martingale $X$, increasing adapted $A$ with $A_{\infty}=1$ and general  $\Phi$. 
It turns out that this is true and simple to prove, although (perhaps surprisingly) this follows not integrating the pathwise BDG inequalities but rather considering the classical ones on an enlarged space; this observation is probably not new, although we include it since were not able to locate a reference in the literature.

    One can ask whether the BDG inequalities hold not only for local martingales, but also for semimartingales which admit an equivalent local martingale measure; the latter processes being of particular importance in mathematical finance, due to the Fundamental Theorem of Asset Pricing (see \cite{DelbSch:94}). 
As proved  in \cite{Sek79, Sek80}, one can exactly characterize the equivalent measures under which the so called `weighted BDG inequalities' hold. Indeed, given $\hat{\P} \sim \P$ let $Z_t:=\E[ d\hat{\P} / d\P \, | \F_t]$, so that\footnote{See \cite[Chapter 2, Theorem 8.3]{JacodShir:02} and  \cite[Chapter 8, Proposition 1.6]{ReYo99}; the fact that $M$ is a local martingale follows from $dM=(Z_{-})^{-1} dZ$, since $Z$ is a martingale.} $Z_t=\exp(M_t - [M]_t/2 )$ for a unique local martingale $M$ with $M_0=0$.  
 Then, as one can read in  \cite[Theorem 3.17 and 3.18]{kaz94}, if the underlying filtration is such that every $\bP$-martingale  is continuous, $M\in BMO(\P)$ iff there exist $c,C$ such that 
\begin{align}
\label{BDGP}
    \hspace{2cm}    	  c \hat{\E}     \sqrt{[X]_{\infty}}\leq \hat{\E} X_{\infty}^* \leq C  \hat{\E}    \sqrt{[X]_{\infty}} \, 
\end{align} 
holds for every local $\P$-martingale $X$. 
    While we cannot use the pathwise BDG inequalities to obtain with a simple proof the above  extremely satisfying result in complete generality, we can easily prove a weaker statement which does not require any knowledge about $BMO$-martingales.
    
 Finally, we briefly discuss what happens to the pathwise and the standard BDG inequalities in higher dimension (finite and infinite).

     The outline of the rest of the paper is then as follows. 
     In Section \ref{sec-DavisSemimart} we  introduce most of the notations and we derive the pathwise Davis inequalities for cadlag semimartingales from their discrete time version. In Section \ref{sec-DavisContinuous} we present an alternative and direct proof for the case of \emph{continuous} semimartingales.     
     In Section \ref{sec-BDGSemimart} we derive the    pathwise BDG inequalities for cadlag semimartingales from Davis' ones.
     In Section \ref{sec-Bessel} we prove the BDG inequalities for the Bessel processes.
     In Section \ref{sec-RandomTime} we show that the BDG inequalities hold for martingales stopped at many random times, and in  
  Section \ref{sec-RandomizedStTime} we discuss \eqref{BDGdA}.
     In Section \ref{sec-Change of Measure} we discuss what happens after a change of measure, and   finally in Section \ref{sec-HigherD} we discuss the multidimensional case.

\section{Pathwise Davis inequalities for cadlag semimartingales}
\label{sec-DavisSemimart}
We now easily obtain a version of the pathwise Davis inequalities for cadlag semimartingales by passing to the limit their discrete time version; before however, let us introduce most of the notations used throughout the paper.  We will work on an underlying filtered probability space $(\Omega, \bF, (\cF_t)_{t\geq 0}, \bP)$ whose filtration $(\F_t)_t$  satisfies the usual conditions\footnote{Meaning it is right continuous and $\cF_0$ contains all the negligible sets of $\cF_{\infty}$.}.
Given cadlag adapted processes $S, X,A$, and assuming that $X$ is a semimartingale, $A$ is of finite variation (on compact sets) and the following integrals exist, we will use the following notations. 
The cag predictable process  $S_-$  has value $S_{t-}:=\lim_{u \uparrow  t} S_u$ at time $t$, the jump of $S$ at $t$ is $\Delta S_t=S_t -S_{t-}$, the running maximum $S^*$ of $S$ is given by $S^*_{t}:=\sup_{u \leq t} |S_u|$, $[X]$ is the quadratic variation of $X$,  
$A_{\infty}$ is the (possibly infinite) limit $\lim_{t\to \infty} A_t$ (which always exists if $A$ is increasing),   $(H\cdot X)_t$ is the stochastic integral  $\int_{(0,t] } H_u dX_u$,  and  $\int_{0}^{t} H_u dA_u $ (resp. $\int_{0-}^{t-} H_u dA_u $) is the Lebesgue-Stieltjes integral  $\int_{(0,t] } H_u dA_u$ (resp. $\int_{[0,t) } H_u dA_u$).
Given arbitrary processes $K,\tilde{K}$ we will write that $K\leq \tilde{K}$ if $K_t\leq \tilde{K}_t$ holds $\bP$ a.s. and for all $t\geq 0$, and we define by convention $K_{0-}:=0$ and  $0/0:=0$, so that in particular  $X_{0-}=X^*_{0-}=[X]_{0-}=0$, the integrands $H$ in Theorem \ref{PathDCadlag} and  $H,G, F_t^{(s)}$  in Theorem \ref{PathBDGCadlag} are well defined and the measure $dX_{\cdot}^*$ (resp. $d\sqrt{[X]_{\cdot}}$) has mass $X^*_0=|X_0|$ (resp. $\sqrt{[X]_0}=|X_0|$) at $0$.

Here come the pathwise Davis inequalities for cadlag semimartingales.
 \begin{theorem}\label{PathDCadlag}
Let $(X_t)_{t\in [0,\infty)}$ be a cadlag semimartingale, and set 
\begin{align}
\label{HDavis}
H_t:= \frac{X_{t-}}{\sqrt{[X]_{t-} +(X_{t-}^*)^2 }} .
\end{align} 
Then $H$ is cag, predictable, has values in $[-1,1]$, and satisfies
\begin{align}\label{eqPathDCadlag}
    \sqrt{[X]}   \leq 3 X^* - (H\cdot X) \,, \qquad          X^* \leq 6    \sqrt{[X]}   +  2(H\cdot X) \, ;
   \end{align} 
  \end{theorem}
\begin{proof}
Applying \cite[Theorem 1.2]{Sio13BDG} to $x_i:=X_{i/2^n}^t(\omega)$, $0\leq i \leq N$ with $2^n t\leq N $ gives
\begin{align}\label{discrDavis}
    Q^n(X)_t:=\sqrt{\sum_{i\in \bN} \big(X^t_{(i+1)/2^n} -X^t_{i/2^n} \big)^2}   \leq 3 M^n(X)_t    - (H^n\cdot X)_t  \, ,
   \end{align} 
where $M^n(X)_t:=\max_{i\in \bN} |X^t_{i/2^n}| $   and
\begin{align}
\label{Hn}
H^n_s:=\sum_{i\in \bN} 1_{(\frac{i}{2^n}, \frac{i+1}{2^n}]}(s)
\frac{X_{\frac{i}{2^n}}}{\sqrt{ Q^n(X)_{\frac{i}{2^n}} +( M^n(X)_{\frac{i}{2^n}})^2 }} .
\end{align} 
Passing to a subsequence (without relabeling) we get that  $Q^n(X)_s\to [X]_s$ 
$\bP$ a.s. for all $s\geq 0$; since $X$ is cadlag,   $ M^n(X)_s\to X^*_s$ and so\footnote{If $[X]_{s-} + (X^*_{s-})^2=0$ then $H^n_s=H_s=0$.} $H^n_s \to H_s$ $\bP$ a.s. for all $s\geq 0$.
  Since $|H^n|\leq 1$,  using the stochastic dominated convergence theorem we can take limits in \eqref{discrDavis} and obtain that    $\bP$ a.s.  $\sqrt{[X]_t}   \leq 3 X^*_t - (H\cdot X)_t $; since $t$ was arbitrary and all the processes involved are cadlag the inequality  also holds $\bP$ a.s. for all $t\geq 0$.
  The proof of the second inequality is identical.
\end{proof}

The traditional Davis inequalities \eqref{DCadlag} are a simple corollary of the pathwise ones.
  \begin{corollary}
\label{TradDavisIneq}
Under the assumptions of Theorem \ref{PathDCadlag}  
\begin{align}
\label{Davisbounds}
(H\cdot X)^*  \leq    3(X^* + \sqrt{[X]} ) \quad \text{ and } \quad [H \cdot X]\leq [X] .
\end{align} 
If $X$ is a local martingale  then so is $(H\cdot X)$, and
\begin{align}\label{DCadlag}
  \E  \sqrt{[X]_{\infty}}   \leq 3 \E X^*_{\infty}  , \qquad  \E X^*_{\infty}  \leq 6 \E \sqrt{[X]_{\infty}}  \, ;
   \end{align} 
 if moreover $  \E  \sqrt{[X]_{\infty}}  <\infty$ then  $(H\cdot X)$ is a \emph{martingale} and $(H\cdot X)^*_{\infty} \in L^1(\P)$.

\end{corollary} 
\begin{proof}
It trivially follows from \eqref{eqPathDCadlag} that 
$-3    \sqrt{[X]}\leq  (H\cdot X)   \leq 3 X^* $; since $|H|\leq 1$, $[H\cdot X]=H^2\cdot [X]\leq [X]$, \eqref{Davisbounds} hold.
Assume now that  $X$ is a local-martingale, in which case also $H\cdot X$ is a  local martingales (because $H$ is cag). 
Let $\tau_n$ be sequence of stopping times which localizes $H\cdot X$; applying \eqref{eqPathDCadlag} to $X^{\tau_n}$ (instead of $X$), taking expectations and then taking limits for $t,n\to \infty$ we get \eqref{DCadlag} by monotone convergence.
If $  \E \sqrt{[X]_{\infty}} $  is finite \eqref{Davisbounds} gives that 
 $\E(H\cdot X)^*_{\infty}<\infty$, so the dominated convergence theorem ensures that the local martingale $H\cdot X$ is a martingale. 
 \end{proof}

\section{Davis inequality for continuous local martingales}
\label{sec-DavisContinuous}
We give here an alternative statement and derivation of pathwise Davis' inequalities for  \emph{continuous} semimartingales; the following treatment builds on \cite[Theorem 5.1]{Sio13BDG}, where the easier of the two inequalities was proved for continuous local martingales starting at zero.
The proof in this section has the advantage of being a relatively straightforward application of Ito's formula.
\begin{theorem}\label{BDGcont} If $X$ is a continuous semimartingale  then $ \bP \text{\, a.s. for all } t\geq 0$
\begin{align*}
  X_t^* - 4 \sqrt{[X]_t} \leq   \Big( \frac{2 X_t}{\sqrt{ [X]_t }\vee X^*_t}\cdot X_t\Big)_t  \leq  3 X_t^* - 2 \sqrt{[X]_t}    .
\end{align*}
\end{theorem}
Notice that the functional form of the integrand in Theorem \ref{BDGcont} is slightly different\footnote{We conjecture that \cite[Theorem 1]{Sio13BDG}, from which Theorem \ref{PathDCadlag} follows, also holds with the integrand $h_n=x_n/(\sqrt{[x]_n} \vee x^*_n)$ and potentially different constants, although proving this would require much longer computations.} from the one obtained in Theorem \ref{PathDCadlag}.  
 The next lemma is just  a slight modification of the arguments after equation (4.3) in \cite{Sio13BDG}.
\begin{lemma}
\label{dfveeg}
Let $f, g:\R^+\to \R^+$ be continuous increasing and such that $f(0)>0$ and $g(0)>0$.
 Then on $\bR^+$
\begin{align} \label{2g-3f}
 2 g- 3 f \leq \frac{g^2}{f\vee g}+ \int_0^{\cdot} \frac{g^2}{f^2\vee g^2}\, d(f\vee g) - \int_0^{\cdot} \frac1{f\vee g}\, df^2\leq 3 g- 2 f \, .
\end{align}
\end{lemma} 
\begin{proof}
  Since $d\frac{1}{x}=-\frac{1}{x^2} dx$, by a change of variables 
 $\int \frac{g^2}{f^2\vee g^2}\, d(f\vee g)= \int g^2\, d\frac{-1}{f\vee g}$ and
  integrating the latter by parts we obtain that the middle term in \eqref{2g-3f} equals 
\begin{align}
\label{g^2-f^2}
\int_{0}^{\cdot} \frac{d(g^2-f^2)}{f\vee g } \,  .
\end{align} 
As easily shown with the arguments\footnote{Unlike \cite{Sio13BDG}, our $f$ and $g$ are strictly positive; however this has  only the effect of slightly simplifying the calculations.} after equation (4.3) in \cite{Sio13BDG},  \eqref{g^2-f^2} is always smaller than $3g-2f$. Applying this inequality with the role of $f$ and $g$ reversed and then multiplying by $-1$ we see that \eqref{g^2-f^2} is bigger than  $-(3f-2g)$.
\end{proof}

\begin{proof}[Proof of Theorem \ref{BDGcont}]
For $\eps\geq 0$ define $H^{\eps}:=\frac{2X}{([X]+\eps)\vee X^*}$ and  $I_t^{\eps}:=(H^{\eps} \cdot X)_t$.
When $\eps>0$ applying Ito's formula to  $X^2$ and 
$1/(([X]+\eps)\vee X^*)$ we find
$$ d\frac{X^2}{ ([X]+\eps)\vee X^*}=\frac{-X^2 d\Big(([X]+\eps)\vee X^*\Big)}{( ([X]+\eps)\vee X^*)^2}  + \frac {2X\, dX+\, d[X] } {([X]+\eps)\vee X^*} \, .$$
In other words  for $\eps>0$ the integral  $I_t^{\eps} $ equals\footnote{We also use the trivial fact that $d[X]=d([X]+\eps)$.}
\begin{align}\label{BDG1tbE}
 \frac{X_s^2}{([X]_s+\eps)\vee X_s^*} \bigg|^{s=t}_{s=0} + \int_0^t \frac{ X^2 d\Big(([X]+\eps)\vee X^*\Big)}{(([X]+\eps)\vee X^{*})^2 }\,  -  \int_0^t \frac {d([X]+\eps)}{([X]+\eps)\vee X^*}    .
\end{align}
Since $X_0^2=X_0^{*2}$, the quantity in \eqref{BDG1tbE} trivially gets bigger if we replace each occurrence of $X$ by $X^*$, so applying Lemma \ref{dfveeg} to  $f(t)= [X]_t+\eps, g(t)= X_t^*$ gives 
$$I^{\eps} \leq 3  X^* - 2([X]+\eps) .$$
To pass to the limit as $\eps\to 0$ notice that if $[X]_t\vee X_t^*=0$ then $X_t=0$, so $H_t^{\eps}=0$ and $H_t^{0}=0$ (since by our definition $0/0=0$); if instead 
$[X]_t\vee X_t^*>0$ then $H_t^{\eps}\to H_t^{0}$ is trivially true.
In summary, the stochastic dominated convergence theorem gives that $I^{\eps}\to I^0$ uniformly on compacts in probability as $\eps\to 0$, so there exists some $\eps_n\to 0$ for which $I^{\eps_n}_t\to I^0_t$ a.s. for all $t$, proving that 
$I^{0} \leq 3  X^* - 2[X] .$

To prove the opposite inequality, replace $X$ with $X^*$ in \eqref{BDG1tbE} in both occurrences, and call $J^{\eps}$ the resulting quantity; then Lemma \ref{dfveeg} applied to  $f= [X]+\eps, g= X^*$ yields $J^{\eps} \geq 2  X^* - 3([X]+\eps) $. The thesis $I^{0} \geq   X^* - 4[X] $ then follows taking $\eps_n\to 0$ as above if we can show that $I^{\eps}- J^{\eps} \geq -X^* - [X]$. To prove the latter, let us bound separately the two  terms $C^{\eps},D^{\eps}$ whose sum gives $I^{\eps}- J^{\eps}$. First 
$$C^{\eps}:=\frac{X_t^2- X_t^{*2}}{([X]_t+\eps)\vee X_t^*}\geq \frac{ 0 - X_t^{*2}}{([X]_t+\eps)\vee X_t^*}\geq - X_t^* . $$
For the second term, notice that if $f,g$ are continuous increasing then 
\begin{align}\label{df and dg}
d(f \vee g) = 1_{\{f\geq g\}} df + 1_{\{f < g\}} dg ;
\end{align}
applying this to $f= [X]+\eps$ and $g=X^*$, and using that $X^2_t=(X_t^*)^2$  holds\footnote{Indeed  $O:=\{t>0: X_t^*> |X_t| \}$ is open in $\bR$, so it can be written as a countable union of open intervals; on each of these  $X_{\cdot}^*$ is constant, so $dX_{\cdot}^*$ is supported by $\bR_+\setminus O=\{t\geq 0: X_t^*=|X_t| \}$.} for $dX_t^*$ a.e. $t$, we get that 
\begin{align*}
 D^{\eps}:= \int_0^t \frac{ X^2 -  X^{*2}}{ (([X]+\eps) \vee X^{*})^2 }\, d\Big(( [X]+\eps)\vee X^* \Big)  =      \int_0^t  \frac{ X^2 -  X^{*2}}{ ([X]+\eps)^2}\,  1_{\{ [X]+\eps \geq  X^* \}}   d  [X] \, , 
\end{align*} 
which is bounded below by $-[X]_t$ since the integrand on the right hand side  is bounded below by  $-1$.
\end{proof}

\section{Pathwise BDG inequalities for cadlag semimartingales}
\label{sec-BDGSemimart}
In this section we modify ideas of A. Garsia to show that the general (pathwise) 
BDG inequalities are a consequence of Davis' ones; this approach was already taken in \cite{Sio13BDG} in the (technically much simpler) discrete-time case when $\Phi(t)=t^p$ for some $p> 1$. 
 For expositions of Garsia's ideas we refer to \cite[Pag 101 to 106]{Me76} or  \cite{Ch75}; a slightly modified version\footnote{The quantity $\E[\Phi(|X|)]$ is replaced by the seminorm 
$||X||_{\Phi}:=\inf \{\lambda>0: \E[\Phi(|X|/\lambda)]\leq 1 \} $.} is given in \cite[Chapter 7, Lemma 91]{DeMeB}.

Notice that we \emph{do not} derive the general pathwise
BDG inequalities passing to the limit the discrete time statement \cite[Theorem 6.3]{Sio13BDG} as done for Davis inequalities. The problem with this approach is that for $p>1$ it is not easy to show that the discretized integrands $H^n, G^n$ corresponding to \eqref{Hn} should converge to their continuous time equivalent, since $H^n_s$ can be written as\footnote{Indeed $Y^n_s=\Phi(\sqrt{Q^n(X)_s})$} $\int_0^s K^n_u dY^n_u$, where also the integrator $Y^n$ depends on $n$ (and similarly for $G^n$).

We will henceforth consider a function $\phi:\bR_+ \to \bR_+$ which is cadlag, increasing, unbounded and  such that $\phi(0)=0$; as usual we define $\phi(0-):=0$, so that in particular $d\phi$ has no atom at zero.
The integral $\Phi(t):=\int_0^t \phi(s) ds$ is a  convex increasing function such that $\Phi(t)/t\to \infty$ as $t\to \infty$. We will also assume that $\Phi(t)$ is `tame', i.e. that there exists some constant $C_{\Phi}$ such that $\Phi(2t)\leq C_{\Phi} \Phi(t)$ for all $t$; equivalently,  
there exists some constant $c_{\phi}$ such that $\phi(2t)\leq c_{\phi} \phi(t)$ for all $t$. 
Such functions $\Phi$ are well studied in connections to  Orlicz spaces, and  are often called `Young functions', although they are also referred to by various other names. In particular the interested reader should consult \cite{KraRut58}, where $\Phi$ would be called an `N-function satisfying the $\Delta_2$-condition'. One can then show that the `exponent\footnote{If $\Phi(t)=t^c$ then $p$ in \eqref{pexp} equals $c$.}'
\begin{align}
\label{pexp}
p:=\sup_{u>0}	\frac{ u \phi(u)}{\Phi(u)} \, ,
\end{align} 
is in $(1,\infty)$.
Moreover if $\psi(t):=\inf\{s: \phi(s)>t\}$ is the cad (so $\psi_{-}(t):=\psi(t-)$ is the cag)  inverse of $\phi$,
and $\Psi(t):=\int_0^t \psi(s) ds$ is the convex conjugate of $\Phi$, the following inequalities hold:
\begin{align}
\label{Young}
uv\leq \Phi(u)+\Psi(v)   ,\quad \Phi(au)\leq a^p \Phi(u) \text{ if } a\geq 1, \quad \Psi(au)\leq  a\Psi(u) \text{ if } a\leq 1 \, , \\ \nonumber
\Psi(s)\leq (p-1) \Phi(\psi(s-)) \,  \text{ and } \,  \psi_{-}(\phi(s))\leq s  \text{ and so } \Psi(\phi(s))\leq (p-1) \Phi(s)
\end{align} 

 \begin{theorem}\label{PathBDGCadlag}
Assume that $(X_t)_{t\in [0,\infty)}$ is a cadlag semimartingale, $\Phi$ is as above and $p$ is given by \eqref{pexp}. 
Let $C$ (resp. $D$) be the cad inverse of $\sqrt{[X]}$ (resp. $X^*$) and
define $c_p:= p(6p)^{p} $,

$$H_t:= \int_{[0,\sqrt{[X]}_{t-})} pF^{(C_s)}_t d\phi(s),
 \qquad G_t:=\int_{[0,X_{t-}^*)} pF^{(D_s)}_t d\phi(s), $$  where $(F^{(s)}_t)_{(s,t):s <t}$ is defined as 
 $$ F^{(s)}_t:=\frac{X_{t-} - X_{s-}}{\sqrt{[X]_{t-} - [X]_{s-} + \sup_{s\leq u < t}(X_{u} - X_{s-})^2}} \, . $$ 
Then  $H$, $G$ are cag predictable and 
\begin{align}\label{PathBGCadlag}
    \Phi(\sqrt{[X]})   \leq c_p \Phi(X^*) - (H\cdot X) \,, \qquad         
     \Phi(X^*) \leq c_p     \Phi(\sqrt{[X]})   +  2(G\cdot X) .
   \end{align} 
Moreover  if $\phi$ is continuous then $H$ and $G$ are lad and 
\begin{align}
\label{HandG2}
    H_t=\int_{[0,t)} pF^{(s)}_t d\phi(\sqrt{[X]_s}) \, ,\quad  G_t=\int_{[0,t)} pF^{(s)}_t d\phi(X_s^*) \, .
\end{align} 
  \end{theorem}
We will derive in  \eqref{HandG+} integral expressions for $H_{t+}$ and $G_{t+}$; these show that if $X$ is continuous then also $H$ and $G$ are continuous.
 Notice that the function $\Phi(t)=t$ does not satisfy the assumptions made on $\Phi$ in Theorem \ref{PathBDGCadlag}; despite of this,  Theorem \ref{PathDCadlag} affords the equivalent of \eqref{PathBGCadlag}.

Of course, given a sequence of real numbers $(x_n)_{n\geq 0}$ and the probability space $\{\hat{\omega}\}$ made of one point, applying  \eqref{PathBGCadlag} (resp. \eqref{eqPathDCadlag}) to $X_t(\hat{\omega}):=\sum_n x_n 1_{[n,n+1)}(t)$ we obtain pathwise BDG inequalities for functions of a real variable, which if $\Phi(t)=t^p$ reduce to \cite[Theorem 6.3]{Sio13BDG} for $p>1$ (resp. to \cite[Theorem 1.2]{Sio13BDG} for $p=1$).

Moreover, if $\phi$ is continuous the integrands $H$ and $G$ in Theorem \ref{PathBDGCadlag} are caglad, so \eqref{PathBGCadlag} are really path-by-path inequalities.  Indeed, for $\bP$ a.e. $\omega$ one can compute $(H \cdot X)_{\cdot}(\omega)$ and $[X]_{\cdot}(\omega):=X^2(\omega) - (2X_{-} \cdot X)(\omega)$ only making use of $H_{\cdot}(\omega)$ and $ X_{\cdot}(\omega)$ by taking limits of Riemann sums computed along appropriate sequences of hitting times (see\footnote{One can apply the cited theorem since $H^+:=(H_{t+})_t$  is adapted (since such is $H$, and the filtration is right continuous) and $H_{t-}^+=H_t$.}  \cite[Theorem 7.14]{Bi81} and \cite{Ka95}). 
Unfortunately this remark does not apply to Davis inequalities \eqref{eqPathDCadlag}, since  $H$ in Theorem \ref{PathDCadlag} is not necessarily lad, not even if $X$ is continuously differentiable\footnote{\label{nlad} If $X_t:=t^2\sin(1/t)$ for $t>0$ and $X_0:=0$, $X$ is $C^1$ and thus has finite variation, so it is a semimartingale with $[X]=0$, and   $H_t=X_{t-}/X^*_{t-}$ keeps oscillating between $1$ and $-1$ as $t\downarrow 0$.}.
  
  The traditional BDG inequalities 
   are a simple corollary of the pathwise ones.  
  \begin{corollary}
\label{HXmart}
Under the assumptions of Theorem \ref{PathBDGCadlag} for all $y\in [\frac{1}{p},1]$
\begin{align}
\label{[HXGX]}
\sqrt{[H\cdot X]} \leq p^2 \Phi(\sqrt{[X]}) , \qquad 
\sqrt{[G\cdot X]} \leq (py)^p \Phi(\sqrt{[X]}) +\frac{p-1}{y} \Phi(X^*) ,
\end{align} 
and  if $X$ is a local-martingale then so are $(H\cdot X)$ and $ (G\cdot X)$, and 
\begin{align}\label{BGCadlag}
  \E \Phi(\sqrt{[X]_{\infty})}   \leq c_p \E \Phi(X^*_{\infty})  , \qquad  \E\Phi(X^*_{\infty})  \leq c_p \E\Phi(\sqrt{[X]_{\infty}}) .
   \end{align} 
In particular if $X$ is a local-martingale and $  \E  \Phi(\sqrt{[X]_{\infty})} <\infty$   then  $(H\cdot X)_t$ and $ (G\cdot X)_t$ are \emph{martingales} and $(H\cdot X)^*_{\infty}, (G\cdot X)^*_{\infty} \in L^1(\P)$.
\end{corollary}

The rest of this section is devoted to the proof of Theorem \ref{PathBDGCadlag} and  Corollary \ref{HXmart}.

\begin{proof}[Proof of Corollary \ref{HXmart}]
Since for increasing positive $A,B$, with cad $B$, we have $\int_0^t A dB\leq A_{t} B_t$, and since $|F_t^{(s)}|\leq 1$ implies that $H^2\leq p^2 \phi(\sqrt{[X]})^2$, using \eqref{pexp} we get 
$$[H\cdot X]=H^2\cdot [X]\leq p^2 \phi(\sqrt{[X]})^2 \cdot [X]
\leq p^2 \phi(\sqrt{[X]})^2  [X]
\leq p^2 (p\Phi(\sqrt{[X]}))^2 .$$

Analogously for $G$ we can write, for any $y>0$
$$\sqrt{[G\cdot X]}\leq \sqrt{p^2 \phi(X^*)^2 \cdot [X]}
\leq p \phi(X^*) \sqrt{[X]} = \big(py\sqrt{[X]}\big) \left(\frac{\phi(X^*)}{y} \right)  ;
$$
when $y\in [1/p,1]$ we can apply the inequalities  \eqref{Young} and get
$$
 \big(py\sqrt{[X]}\big) \left(\frac{\phi(X^*)}{y} \right)  
 \leq
  \Phi(py\sqrt{[X]})+ \Psi\left(\frac{\phi(X^*)}{y} \right) \leq
    (py)^p\Phi(\sqrt{[X]})+ \frac{\Psi(\phi(X^*))}{y} ;
$$
now bound the last term above using  that $\Psi(\phi(s))\leq (p-1) \Phi(s)$; putting the inequalities together concludes the proof of \eqref{[HXGX]}.

If  $X$ is a local-martingale, working as in Corollary \ref{TradDavisIneq} gives the thesis: the only difference here is that $\E(H\cdot X)^*_{\infty}<\infty$ follows 
from $  \E  \Phi(\sqrt{[X]_{\infty})}<\infty$ since \eqref{[HXGX]} gives that 
 $  \E  \sqrt{[H\cdot X]_{\infty}} <\infty$ and we can then apply  \eqref{DCadlag}  to $H\cdot X$ (instead of $X$).
 \end{proof}

\begin{proof}[Proof of Theorem \ref{PathBDGCadlag}.] 
\textbf{Step 1: $H,G$ are cag predictable.} 
Since  $\{C_s < t \}=\{s < \sqrt{[X]}_{t-}\}$,  $F_t^{(C_{\cdot})}$ is defined on $[0,\sqrt{[X]}_{t-})$ and setting $F^{(s)}_t:=0$ for $s\geq t$ we get\footnote{We write   $\int_0^{\infty}$ for  $\int_{(0,\infty)}$, which is the same as  $\int_{(0,\infty]}$ since by definition $\Phi(\infty)=\lim_{t\to \infty} \Phi(t)$.} 
$$H_t:= \int_{0}^{\infty} pF^{(C_s)}_t d\phi(s),
 \qquad G_t:=\int_{0}^{\infty} pF^{(D_s)}_t d\phi(s). $$ 
 Denote with $\F_t \otimes \mathcal{B}$  the product sigma algebra of $\F_t$ with the Borel sets $\mathcal{B}$ of $\bR_+$; we will now prove that $(F_t^{(C_{s})})_{s\in [0,\infty)}$  is $\cF_t\otimes \cB$ measurable,  so that  $H_t$ is  $\cF_t$ measurable, i.e. $H$ is adapted. 
 Since $\sqrt{[X]}$ is adapted, $C_s$ is a stopping time and so $C_s\wedge t$ is $\cF_t$ measurable, and so since $C_{\cdot}$ is cad the map $Z(\omega,s):= (\omega, C_s(\omega)\wedge t)$ is $\cF_t\otimes \cB / \cF_t\otimes \cB$ measurable. Analogously (but using \emph{left} continuity\footnote{We warn the reader that for  $s\uparrow t$ the limit of $F^{(s)}_t$ may not exist (so we specified $s\in [0,t)$).} of $(F_t^{(s)})_{s\in [0,t)}$ and the fact that  $F^{(s)}_t=0$ for $s\geq t$) the map $F^{(\cdot)}_t$ is  $\cF_t\otimes \cB$ measurable, and thus so is the composition $F^{(C_{\cdot})}_t=F^{(\cdot)}_t \circ Z$.
  Since $F^{(C_s)}_t\in [-1,1]$, the dominated convergence theorem implies that $H$ is cag (so predictable):  indeed if $t_n\uparrow t$ and $M^{C_s}_t:=\sup_{C_s\leq u <t} |X_{u} -X_{C_s -}|>0$   then trivially $M^{C_s}_{t_n}\to M^{C_s}_t$ and  $F^{(C_{s})}_{t_n}\to F^{(C_{s})}_t$, whereas if $M^{C_s}_t=0$ then trivially $M^{C_s}_{t_n}=0$ and so   $0=F^{(C_{s})}_{t_n}\to F^{(C_{s})}_t=0$. 
  The proof that $G$ is adapted and cag  is analogous.
  
  \textbf{Step 2: $X$ satisfies $  \Phi(\sqrt{[X]})   \leq c_p \Phi(X^*) - (H\cdot X)$.} 
Since $C_s$ is a stopping time $Y^{(s)}_t:=X_{C_s+t} -X_{C_{s}-}$ is a  semimartingale (w.r.t. the time changed filtration $\mathcal{F}^{(s)}_t:=\mathcal{F}_{C_s+t}$) and satisfies
\begin{align}
\label{[Ys]}
 \sqrt{[X]_{C_s+t}}-  \sqrt{[X]_{C_s -}} \leq  \sqrt{[X]_{C_s+t} -  [X]_{C_s -} }
 =\sqrt{[Y^{(s)}]_t} 
\end{align} 
Define  $f(x,q,s):=x/\sqrt{q+s^2}$ and apply \eqref{eqPathDCadlag} to $Y^{(s)}$ 
 (instead of $X$) to get
\begin{align} \label{PDavY}
    \sqrt{[Y^{(s)}]}   \leq 3 Y^{(s)*} - (H^{(s)}\cdot Y^{(s)}) \quad \text{ where } \quad 
    H^{(s)}:=f(Y_{-}^{(s)},[Y^{(s)}]_{-},Y^{(s)*}_{-})   \, .
   \end{align} 
Writing the integrals as limits (in probability, uniformly on compacts) of Riemann sums we get 
\begin{align*} \textstyle
(H^{(s)}\cdot Y^{(s)})_{\cdot}=\int_{C_s}^{C_s+ \cdot}F_u^{(C_s)} dX_u  \, ,
\end{align*} 
so since $Y^{(s)*}_t \leq 2 X^*_{C_s+t}$, using \eqref{[Ys]},  \eqref{PDavY} we get
\begin{align} \textstyle
\label{d[X]leq}
 \sqrt{[X]_{C_s+\cdot}}-  \sqrt{[X]_{C_s-}} \leq  6  X^*_{C_s+\cdot} -
  \int_{C_s}^{C_s+\cdot}F_u^{(C_s)} dX_u \, .
\end{align} 
Since $\sqrt{[X]}_{C_s -}\leq s$,  $\{ \sqrt{[X]_t} \geq s \} = \{ C_{s-} \leq  t \} \supseteq \{ C_{s} \leq  t \}$ and $[X]$ is constant on $[C_{s-},C_s)$ (when this interval is non-empty) we get
\begin{align*}
(\sqrt{[X]_t}-s)^+ \leq  (\sqrt{[X]_t} - \sqrt{[X]_{C_s -}}) 1_{\{ \sqrt{[X]_t} \geq s \}}  = (\sqrt{[X]_t} - \sqrt{[X]_{C_s -}}) 1_{\{ C_{s} \leq  t \}}  
\end{align*} 
which, combined with \eqref{d[X]leq} and with $X^*_{t}  1_{\{ C_{s} \leq  t \}}\leq X^*_{t}  1_{\{ \sqrt{[X]_t} \geq s \}}  $,     gives
\begin{align}
\label{<lambda}
(\sqrt{[X]_t}-s)^+  \leq  
 1_{\{ \sqrt{[X]_t} \geq s \}}  6  X^*_{t} - 1_{\{ C_{s} \leq  t \}}   \int_{C_s}^{t}F_u^{(C_s)} dX_u  .
\end{align} 
Integrating \eqref{<lambda} over $s\in [0,\infty)$ with respect to $d\phi(s)$ and using the identities  
$\Phi(t)=\int_0^{\infty} (t-s)^+ d \phi(s)$ and 
$1_{\{ C_{s} \leq  t \}}   \int_{C_s}^{t}F_u^{(C_s)} dX_u =  
\int_{0}^{t}F_u^{(C_s)} 1_{\{ C_{s} <  u \}} dX_u $   gives
\begin{align} 
\label{AfterFub}
\Phi(\sqrt{[X]_t})\leq  6  X^*_{t}\phi(\sqrt{[X]_t}) -
\int_{0}^{\infty} \Big( \int_{0}^{t}F_u^{(C_s)} 1_{\{ C_{s} <  u \}} dX_u \Big)  d\phi(s)
\end{align} 
Since  $\{C_s<u \}=\{s<\sqrt{[X]}_{u-}\}$,   the stochastic Fubini theorem \cite[Chapter 4, Theorem\footnote{As observed after the statement of the theorem, by passing from $d\mu$ to  $f d\mu$ for some $f\in L^1(\mu)$ one can prove the theorem for \emph{sigma}-finite $\mu$. Saying it differently, one can consider the \emph{finite} measure $d\mu(s)= \exp(-\phi(s)) d\phi(s)$ and apply Theorem 65 to  $H_s^u:=\exp(\phi(s)) F_u^{(C_s)} 1_{[0,A_{u-})}(s)$.} 65]{Pro04} gives
 \begin{align} \label{Fubini}
\int_{0}^{\infty}  \int_{0}^{t}F_u^{(C_s)} 1_{\{ C_{s} <  u \}} dX_u  d\phi(s) =
\int_{0}^{t}  \int_{[\sqrt{[X]}_{0-},\sqrt{[X]}_{u-})} F^{(C_s)}_u d\phi(s)  dX_u .
\end{align} 
Now apply the inequalities \eqref{Young} to write 
\begin{align*}
 6X^*_{t} \phi(\sqrt{[X]_t})\leq
  \Phi(6pX^*_{t}) +\Psi(\phi(\sqrt{[X]_t})/p)\leq 
  (6p)^p\Phi(X^*_{t}) +\Psi(\phi(\sqrt{[X]_t}))/p
\end{align*} 
and bound the last term using  that $\Psi(\phi(s))\leq (p-1) \Phi(s)$; combine the resulting inequality with \eqref{AfterFub} and \eqref{Fubini} to get 
\begin{align*} 
\Phi(\sqrt{[X]_t})\Big(1-\frac{p-1}{p}\Big) \leq   (6p)^p  \Phi(X^*_{t}) -  
\int_0^t \frac{H_u}{p}  dX_{u} ,
\end{align*} 
i.e. the first inequality \eqref{PathBGCadlag}. 

  \textbf{Step 3: $X$ satisfies $\Phi(X^*) \leq c_p     \Phi(\sqrt{[X]})   +  2(G\cdot X) $.} 
Proceeding analogously using $X^*$ and $D$  (instead of $\sqrt{[X]}$ and $C$) yields a $Y^{(s)}$ which satisfies
\begin{align*}
X_{D_s+t}^* -  X_{D_s-}^* \leq \sup_{u\in [D_s,D_s+t]} |X_u - X_{D_s-}|=  Y^{(s)*}_t \end{align*} 
and since $\sqrt{[Y^{(s)}]_{\cdot}}  \leq  \sqrt{[X]_{D_s + \cdot}}$ we obtain 
\begin{align*} \textstyle
X_{D_s+ \cdot}^* -  X_{D_s -}^* \leq 6   \sqrt{[X]_{D_s + \cdot }}  +  2     \int_{D_s}^{D_s+\cdot}F_u^{(D_s)} dX_u \, ;
\end{align*} 
 the proof continues  exactly as before, yielding the second inequality \eqref{PathBGCadlag}. 
 
   \textbf{Step 4: Alternative expression for $H,G$.} 
 If $\phi$ is continuous and $A=[X]$, the cad inverse $g$ of $\phi \circ A$ equals $C \circ \psi $ and  $(\phi\circ A)(u-)=\phi(A_{u-}-)$, so 
 \begin{align}
\label{phicont}
\int_{(\phi\circ A)(0-)}^{(\phi\circ A)(u-)} F^{(g(s))}_u ds  =
\int_{\phi(A_{0-}-)}^{\phi(A_{u-}-)} F^{(C_{\psi(s)})}_u ds
\, .
\end{align} 
By change of variable the first  integral in \eqref{phicont} equals 
  $\int_{[0,u)} F^{(s)}_u d\phi(A_s)$, and  the second one  $\int_{[A_{0-},A_{u-})} F^{(C_s)}_u d\phi(s)$; thus  $\int_{[0,t)} pF^{(s)}_t d\phi(\sqrt{[X]_s})$ equals
   $H_t$. Proceed analogously for $A=X^*$ and $G$.

\textbf{Step 5: $H$ and $G$ are lad.}
We will use the expression \eqref{HandG2}, writing however the integrals over $[0,\infty)$ and extending $F^{(s)}_t$ to be zero for $s\geq t$. Define the quantities $\hat{M}^{s}_t:=\sup_{s\leq u \leq t} |X_{u} -X_{s -}|$ for $s \leq t$ (so that $\hat{M}^{s}_t=|X_{t} -X_{t-}|$ for $s=t$) and
 $$ \hat{F}^{(s)}_t:=\frac{X_{t} - X_{s-}}{\sqrt{[X]_{t} - [X]_{s-} + (\hat{M}^{s}_t)^2}} \, \quad \text{ for } \, s\leq t,
 \quad   \hat{F}^{(s)}_t:= 0 
 \quad \text{ for } \, s> t \, .$$ 
If $t_n\downarrow t$ and $s>t$    then by definition $F^{(s)}_{t_n}=0=\hat{F}^{(s)}_t$ for $n$ big enough (such that $t_n <s$).
If $s\leq t$ then trivially $\hat{M}^{s}_{t_n}\to \hat{M}^{s}_t$, so if $\hat{M}^{s}_t>0$ we get  $F^{(s)}_{t_n}\to \hat{F}^{(s)}_t$; however  if $\hat{M}^{s}_t=0$ it can happen\footnote{For example take $s=0$ and see footnote \ref{nlad}.} that $\hat{M}^{s}_{t_n}>0$ for all $n$ and that  $F^{(s)}_{t_n}$ does not converge to $\hat{F}^{(s)}_t=0$. We shall now show that this can only happen for $s$ in  a set of $d\mu:=d(\phi(\sqrt{[X]_{\cdot}})+\phi(X^*_{\cdot}))$ measure zero, so $F^{(s)}_{t_n}\to \hat{F}^{(s)}_t$ for $d\mu$ a.e. $s$ and the dominated convergence theorem implies that  $H$ and $G$ lad and
\begin{align}
\label{HandG+}
    H_{t+}=\int_{[0,t]} p\hat{F}^{(s)}_t d\phi(\sqrt{[X]_s}) \, ,\quad  G_{t+}=\int_{[0,t]} p\hat{F}^{(s)}_t d\phi(X_s^*) \, .
\end{align} 
If the set $Z:=\{s: \hat{M}^{s}_t=0\}$ contains some element $\underline{s}$ then necessarily it contains the whole interval $[\underline{s},t]$; it follows that, for some  $s_0\in [0,t]$, $Z$ equals either  $[s_0,t]$ or  $(s_0,t]$. Since $X$ is cad, if $Z\ni s_n\downarrow s$ then $s\in Z$, so $Z=[s_0,t]$.
Since $X_u$ takes the constant value $X_{s_0 -}$ for all $u\in [s_0,t]$, 
  $[X]_u$ (resp. $X_u^*$) takes the constant value $[X]_{s_0 -}$ (resp. $X^*_{s_0 -}$) for all $u\in [s_0,t]$, so  $d\mu$ gives measure $0$ to $[s_0,t]$.

  \end{proof}

\section{The Bessel process}
\label{sec-Bessel}
  
  We will now prove BDG inequalities for the Bessel process as a corollary of the pathwise Davis inequalities. While Davis inequalities follow easily, to recover the general BDG inequalities we do not use Theorem \ref{PathBDGCadlag}, but rather apply   a strengthened version of Davis inequality,  obtained by a rather delicate  modification of the arguments in \cite[Chapter 7, Lemma 91]{DeMeB}. 
 \begin{theorem}
\label{Bessel}
Let $X$ be the  Bessel process of dimension $\alpha \in [1,\infty)$  started at $X_0 \geq 0$ and $\Phi(t)=t^p$ for $p>0$, then there exist constants $c,C$ such that
\begin{align}
\label{BDGBessel}
c\, \E \Phi(\sqrt{[X]_{\tau}})  \leq \E \Phi(X_{\tau}^*)   \leq C \, \E \Phi(\sqrt{[X]_{\tau}})  
\quad \text{ for all  stopping times $\tau$ } .
      	 \end{align}
More generally, \eqref{BDGBessel} holds if $\Phi:\bR_+ \to \bR_+$ is cadlag, increasing,  such that $\Phi(x)=0$ iff $x=0$ and for which $\sup_{t>0}\Phi(\beta t)/\Phi(t) < \infty $ for some (and thus all) $\beta>1$.
\end{theorem}

The only facts about $X$ which we will need in the following proof 
 is that $X_{\cdot}>0$ $\bP\otimes \mathcal{L}^1$ a.e. and $X$ is a  weak solution\footnote{The explosion time  of \eqref{bessel} is $\infty$, i.e. the solution to 	\eqref{bessel} is defined for all $t\in [0,\infty)$.} 
 of      \begin{align}
\label{bessel}
dX=\frac{\alpha -1}{2X} dt + dW , 
\end{align} 
where $W$ is a standard Brownian motion w.r.t some underlying filtration $(\cF_t)_t$.
In fact $X$  is positive and it never hits $0$ (after time zero) if $\alpha \geq 2$, whereas for $\alpha \in (0,2)$  a.s. $X$ hits zero but the set $\{s: X_s=0 \}$ has Lebesgue measure zero (for all these statements see \cite[Page 442]{ReYo99}).
That $X$ solves  \eqref{bessel} is stated in \cite[Chapter 11, Exercise 1.26]{ReYo99}, and if $\alpha \geq 2$  the proof is simple and can be found just before \cite[Chapter 11, Proposition 1.10]{ReYo99}.

 Notice in particular that $X$ has continuous paths and 
 $[X]_t=X_0^2+ t$.

\begin{lemma}
\label{BesselEstHX}
Let $X$ be a semimartingale such that $X_{\cdot}+C>0$ $\bP\otimes \mathcal{L}^1$ a.e. and  
\begin{align}
\label{besselX_0}
dX_{t}=\frac{\gamma}{2(X_{t}+C)} dt + dW_t  \,  , \quad  X_0\geq 0 \, ,
\end{align} 
where $W$ is a standard Brownian motion w.r.t some underlying filtration $(\cF_t)_t$, 
 $C\geq 0$ is a $\cF_0$ measurable random variable and  $\gamma>0$. Then 
\begin{align}
\label{CBessel1}
 \E ( X^*_{\tau} |\cF_0  )  \leq (6  + 2\gamma)  \E(\sqrt{[X]}_{\tau} |\cF_0 ) \, 
\end{align} 
and if $C=0$ then
\begin{align}
\label{CBessel2}
  \E (\sqrt{[X]}_{\tau} |\cF_0  )  \leq  3 \E( X^*_{\tau} |\cF_0  ) \, .
\end{align} 
\end{lemma} 

    \begin{proof}
 Taking $H_t:= X_t/ \sqrt{[X]_t + (X_t^*)^2}  $, since $|H|\leq 1$ we get that $\E \int_0^t H_s^2 ds \leq t<\infty$, so $H\cdot W$ is a martingale and  \eqref{besselX_0} gives 
 \begin{align}
\label{HdBesselC}
   \E((H\cdot X)_{\tau \wedge t}|\cF_0 )= \E\left(\int_0^{\tau \wedge t}  \frac{\gamma}{2 \sqrt{[X]_s + (X_s^*)^2} }  \frac{X_s}{X_s +C} ds   \Big|\cF_0 \right) .
\end{align} 
In particular, since $[X]_t=X_0^2+ t$  and $X_{\cdot}/(X_{\cdot}+C)\leq 1$ $\bP\otimes \mathcal{L}^1$ a.e., deleting the positive $X^*$ term from \eqref{HdBesselC} we
 can bound $ \E((H\cdot X)_{\tau \wedge t} |\cF_0 )  $ from above with
 \begin{align*}
  \E \left(\int_0^{\tau \wedge t}  \frac{\gamma}{2 \sqrt{X_0^2+s} } ds \Big|\cF_0  \right)=    \gamma\big(\E( \sqrt{X_0^2+\tau \wedge t} |\cF_0 ) -|X_0|\big) \leq 
   \gamma \E( \sqrt{[X]_{\tau \wedge t}} |\cF_0 )  \, .
\end{align*}  
Now evaluate the second inequality \eqref{eqPathDCadlag}  at time $\tau \wedge t$,
take $\E( \ldots | \cF_0)$, apply the bound we proved  for $ \E((H\cdot X)_{\tau \wedge t}|\cF_0 )  $  and  take limits\footnote{Use the monotone convergence theorem.} as $t\to \infty$ to  get \eqref{CBessel1}.   
If $C=0$ then $X_{\cdot}=X_{\cdot}-C>0$ $\bP\otimes \mathcal{L}^1$ a.e., so trivially from \eqref{HdBesselC} we get the bound $ \E((H\cdot X)_{\tau \wedge t}|\cF_0 )\geq 0$, so \eqref{CBessel2}  follows evaluating \eqref{eqPathDCadlag} at time $\tau \wedge t$ and taking  $\lim_{t\to \infty }\E( \ldots | \cF_0)$.
\end{proof}

\begin{proof}[Proof of Theorem \ref{Bessel}.]
The case  $\alpha\in \bN\setminus\{0\}$ follows from the analogous statement for Brownian motion (for which we refer to \cite[Page 37]{LenLepPra80} for general\footnote{The statement in the special case $\Phi(t)=t^p$ can be found in most books on stochastic calculus} $\Phi$ and dimension one; for higher dimension see our Section \ref{sec-HigherD}): let us show this in detail.
If $n\in \bN\setminus\{0\}$, $B$ is a $n$-dimensional Brownian motion started at $B_0$ and $X:=||B||_{\bR^n}$ then $[X]_t=||B_0||^2_{\bR^n}+ t$ and $ [B]_t= ||B_0||^2_{\bR^n}+nt$, so $[B]/n \leq [X] \leq [B]$. Thus, since $X^*_t=\sup_{s\leq t} ||B_s||_{\bR^n}$, the BDG inequalities applied to $B$ imply those for $X$.

From now on we can then assume $\alpha>1$. Since $X$ solves \eqref{bessel}, we can apply Lemma \ref{BesselEstHX} to $X$ with $C=0$ and $\gamma=\alpha -1$,  and taking expectations gives the thesis for $\Phi(t)=t$. So, let us consider the case of general $\Phi$ and $\alpha>1$.  If $\sigma \leq \hat{\sigma}$ are \emph{finite} stopping times then $\hat{W}_t:=W_{\sigma+t}- W_{\sigma}$ (resp.  $\hat{\sigma}-\sigma$) is a $ \hat{\cF}_t:=\cF_{\sigma+t}$ Brownian motion (resp. stopping time),  and 
$\hat{X}_t:=X_{\sigma + t} $ trivially satisfies $d\hat{X}=\frac{\alpha -1}{2\hat{X} } dt + d\hat{W} .$
 We can then apply \eqref{CBessel2} with $(X,W,\cF,\tau,C,\gamma):=(\hat{X},\hat{W}, \hat{\cF},\hat{\sigma}-\sigma,0,\alpha-1)$  and combine this with the bounds $\hat{X}_t^*\leq X^*_{\sigma +t}$ and 
\begin{align*}
 \sqrt{[X]_{\sigma+t}}-  \sqrt{[X]_{\sigma} 1_{\{\sigma>0 \}}}  \leq  \sqrt{[X]_{\sigma+t} -  [X]_{\sigma}1_{\{\sigma>0 \}} }
 =\sqrt{[\hat{X}]_t - X_{\sigma}^2 1_{\{\sigma>0 \}}} \leq  \sqrt{[\hat{X}]_t } \, 
\end{align*} 
to obtain
\begin{align}\label{condbesseldavis}
     \E ( \sqrt{[X]_{\hat{\sigma}}}-  \sqrt{[X]_{\sigma}} 1_{\{\sigma>0 \}} |\cF_{\sigma}  )  
             \leq   3 \E (\hat{X}^*_{\hat{\sigma} -\sigma} |\hat{\cF}_{0}  )   \, 
             \leq   3 \E (X^*_{\hat{\sigma}} |\cF_{\sigma} )   \, .
   \end{align}    
Define the localizing sequence $\sigma_n:=\inf\{t\geq 0: X^*_t + [X]_t\geq n\} \wedge n$   and, given  stopping times $\tau,\hat{\tau}, \theta$ with $\tau \leq \hat{\tau}$, define  $\hat{\sigma}:=\sigma_n  \wedge \theta \wedge  \hat{\tau}$, $\sigma:=\sigma_n    \wedge \theta \wedge \tau$. Since $\sigma \leq \hat{\sigma}<\infty$ and 
$\sqrt{[X^{\sigma_n  \wedge \theta}]}_{\tau}   1_{\{\sigma >0 \}} \leq \sqrt{[X^{\sigma_n  \wedge \theta}]}_{\tau}  1_{\{\tau >0 \}} $
we can apply \eqref{condbesseldavis} to get 
$$     \E ( \sqrt{[X^{\sigma_n  \wedge \theta}]}_{\hat{\tau}}-  \sqrt{[X^{\sigma_n  \wedge \theta}]}_{\tau} 1_{\{\tau>0 \}} |\cF_{\sigma}  )  \leq 
3 \E (X^*_{\hat{\sigma}} |\cF_{\sigma} )   \, .$$
Since $\sqrt{[X^{\sigma_n  \wedge \theta}]}_{\hat{\tau}}-  \sqrt{[X^{\sigma_n  \wedge \theta}]}_{\tau} 1_{\{\tau>0 \}}$ and $X^*_{\hat{\sigma}}$ are 
$\cF_{\sigma_n    \wedge \theta}$ measurable 
and $  \E (Y| \cF_{\sigma})=  \E( \E(Y| \cF_{\sigma_n    \wedge \theta}) | \cF_{\tau}) $
  for any positive r.v. $Y$, we obtain that 
$$\E ( \sqrt{[X^{\sigma_n  \wedge \theta}]}_{\hat{\tau}}-  \sqrt{[X^{\sigma_n  \wedge \theta}]}_{\tau}   1_{\{\tau>0 \}} |\cF_{\tau}  )               \leq   3 \E( X^*_{\hat{\sigma}} |\cF_{\tau} )   \, . $$ 
Integrating this over $\{\hat{\tau}> \tau \}\in \cF_{\tau}$ gives, since 
$\sqrt{[X^{\sigma_n  \wedge \theta}]}_{\hat{\tau}}   1_{\{\hat{\tau}>0 \}} \leq \sqrt{[X^{\sigma_n  \wedge \theta}]}_{\hat{\tau}}  $,
$$\E ( \sqrt{[X^{\sigma_n  \wedge \theta}]}_{\hat{\tau}}   1_{\{\hat{\tau}>0 \}} -  \sqrt{[X^{\sigma_n  \wedge \theta}]}_{\tau}   1_{\{\tau>0 \}} )               \leq   3 \E( (X^{\sigma_n  \wedge \theta})^*_{\hat{\tau}} 1_{\{\hat{\tau}> \tau \}})   \, . $$
It follows from \cite[Lemma 1.1]{LenLepPra80} that  for some constant $c$
$$     \E \Phi(\sqrt{[X]_{\sigma_n\wedge \theta}})= \E \Phi(\sqrt{[X^{\sigma_n  \wedge \theta}]}_{\infty}) \leq   c \E \Phi(X^{\sigma_n  \wedge \theta})^*_{\infty}  = c \E \Phi(X^*_{\sigma_n  \wedge \theta})    \, , $$
and so 
 taking $n\to \infty$ gives 
$     \E \Phi(\sqrt{[X]_{\theta}}) \leq  c \E \Phi(X^*_{ \theta})    \, . $

To prove the opposite inequality unfortunately we cannot work analogously with $\hat{X}$, because we would need to use that $  \sqrt{[\hat{X}]_t }   \leq \sqrt{[X]_{\sigma+ t} }  $, which is not true (indeed $[\hat{X}]_t=[X]_{\sigma+t} -  [X]_{\sigma} +X^2_{\sigma}$).
We consider instead\footnote{Analogously we cannot consider $Y$ instead of $\hat{X}$ in the previous proof, since the corresponding inequality in Lemma \ref{BesselEstHX}  only holds for $C=0$.}  $Y_{t}:=X_{\sigma +t }- X_{\sigma} 1_{\{\sigma>0 \}}$, so that 
\begin{align}
\label{dY}
dY=\frac{\alpha -1}{2(Y+X_{\sigma} 1_{\{\sigma>0\}})} dt + d\hat{W} .
\end{align} 
 Applying  \eqref{CBessel1} with $(X,W,\cF,\tau,C,\gamma):=(Y,\hat{W}, \hat{\cF},\hat{\sigma}-\sigma, X_{\sigma} 1_{\{\sigma>0\}},\alpha -1)$, and combining this with the bounds $  \sqrt{[Y]_t }   \leq \sqrt{[X]_{\sigma+ t} }  $ and 
$$ X^*_{\sigma+t} - X^*_{\sigma} 1_{\{\sigma>0\}} \leq Y^*_{t}  $$
gives, for  $m:=(6  + 2(\alpha-1)) $,
\begin{align}\label{condbesseldavis2}
     \E ( X^*_{\hat{\sigma}} - X^*_{\sigma} 1_{\{\sigma>0\}} |\cF_{\sigma}  )  \leq  
      m \E (\sqrt{[Y]_{\hat{\sigma}-\sigma}} |\hat{\cF}_{0} )  \leq
             m \E ( \sqrt{[X]_{\hat{\sigma}}} |\cF_{\sigma} ) \, ,
   \end{align}  
which is the equivalent of \eqref{condbesseldavis} (with $X^*$ and $\sqrt{[X]}$ reversed).
The proof now continues exactly as above.
\end{proof}

\section{Random times}
\label{sec-RandomTime}
From Theorem \ref{PathBDGCadlag} and Corollary \ref{HXmart} (resp. Theorem \ref{PathDCadlag} and Corollary  \ref{TradDavisIneq}) it follows   that the BDG inequalities \eqref{BDGtau}
hold for any pseudo stopping time $\tau$ and local martingale $X$ if $\Phi$ is as in Theorem \ref{PathBDGCadlag} (resp. if $\Phi(t)=t$); to see this, one first has to localize $X$ so as to make $H\cdot X$ and $G\cdot X$ in $H^1$, then take expectations and then limits (using the monotone convergence theorem).

Although the above extension to pseudo stopping times had already been proved in \cite[Proposition\footnote{This has a typo: $p$ should be $\geq 1$. Only if $M$ is continuous one can take any $p>0$.} 2]{Nik08} with change of filtration techniques, it is convenient that it follows automatically from our approach; moreover, as mentioned in the introduction, we are able to obtain yet another setting in which \eqref{BDGtau} holds, and this seems to be new.
Indeed, one can go the other way around and, given an arbitrary random time $\tau$, study the subspace $\cS_1(\tau)$ of  $M\in H^1$ for which $\E M_{\tau}=0$. The above discussion shows that if\footnote{Because of Corollary \ref{TradDavisIneq} it is clear that $H \cdot X$ is in $H^1$ if one (and thus both) of the quantities in \eqref{DCadlag} are finite; what it not clear is whether $\E (H \cdot X)_{\tau}=0$.} $H \cdot X\in \cS_1(\tau)$, where $X$ is a local martingale and $H$ is as in Theorem \ref{PathDCadlag}, then \eqref{BDGtau} hold for $\Phi(t)=t$; analogously for $H,G$ from Theorem \ref{PathBDGCadlag} and a correspondingly general $\Phi$. As we already said, this can be useful since $\cS_1(\tau)$ is `large' and can be quite explicitly characterized: see \cite[Section 3]{Nik07}. This works out particularly well when $\tau$ is an honest time; for example one can show that 
if   $(\cF_t)_t$ is the filtration generated by a one dimensional Brownian motion $B$ and  $\tau:=	\sup \{ t<\sigma: B_t=0\}$, where $\sigma$ is the first time $B$ hits $1$, 
then $\tau$ is an honest time and if $(L_t)_t$ denotes the local time at zero of $B$ and 
$\cL^n(x)$ is the Laguerre polynomial $\frac{e^x}{n!}\frac{d^n}{dx^n} (x^n e^{-x})$ then $M_t:=\E[\cL_n(L_{\sigma})|\cF_t]$ is in $\cS_1(\tau)$ whenever $n\neq 1$: see \cite[Example 3.7]{Nik07}.

\section{Randomized stopping times}
\label{sec-RandomizedStTime}
In this section we prove that for $\Phi$ as in\footnote{If $X$ is continuous then one can even take $\Phi$ as in Theorem \ref{Bessel}} Theorem \ref{PathBDGCadlag} (and also for $\Phi(t)=t$) we have
\eqref{BDGdA} for any local martingale $X$ and randomized stopping time $A$; but first, we need some more definitions. 
We say that $A$ is a randomized stopping time if it is a cadlag increasing adapted process with $A_0=0$ and $\lim_{t\to \infty} A_t\leq 1$. Breaking from our conventions, in this section we allow $A$ to have a jump at infinity: we will write $A_{\infty-}$ for $\lim_{t\to \infty} A_t$, and we define $A_{\infty}:=1$ (however $X^*_{\infty}$ is defined as usual as $\lim_{t\to \infty} X^*_t$, and analogously for $[X]_{\infty}$) and if $f:[0,\infty]\to [0,\infty] $ is Borel then  $\int_0^{\infty} f(s) d A_s :=\int_{(0,\infty]} f(s) d A_s$.
If $\tau$ is a stopping time (not necessarily finite) $A_t:=1_{[\tau,\infty)}(t)$ is a randomized stopping time, and $f(\tau)=\int_0^{\infty} f(s) d A_s $.

Let us now prove \eqref{BDGdA}; as we just said, the integrals are over $(0,\infty]$.
If the underlying space is $(\Omega, (\cF_t)_t, \bP)$, consider the enlargement $\bar{\Omega}:=\Omega \otimes [0,1]$ endowed with the product probability $ \bar{\bP}:=\bP \otimes \mathcal{L}^1$, and the usual augmentation  $(\bar{\cF}_t)_t$ of the filtration $(\cF_t \otimes \cB)_t$ (where $ \mathcal{L}^1$ denotes the Lebesgue measure and $\cB$ the Borel subsets of $[0,1]$). 
If $\tau_n$ is a  localizing sequence for $X$ and $Y$ a process on $\Omega$, we extend them to $\bar{\Omega}$ by setting for all $t$
 $$\bar{\tau}_n(\omega,s):=\tau_n(\omega) \, , \quad \bar{Y}_t(\omega,s):=Y_t(\omega) \quad \text{ for all } \, \omega\in \Omega, s\in [0,1] .$$ Let  $C(\omega,s):=\inf \{ t: A_t(\omega) > s\}  $ be the cad inverse of $A$; then trivially $\bar{\tau}_n$ are $(\bar{\cF}_t)_t$  stopping times localizing the $(\bar{\cF}_t)_t$  local martingale $\bar{X}$,  and   $C$ is a   $(\bar{\cF}_t)_t$  stopping time  since $\{C< t\}=\{(\omega,s):s < A_{t-}(\omega)\}$ is in  $\cF_t \otimes \cB$.
 In particular we can apply the BDG inequalities \eqref{DCadlag} and \eqref{BGCadlag} to the local martingale $\bar{X}$ stopped at $C$, and then use the  change of time  formula
 $$ \textstyle 
\bar{\E} \bar{Y}_{C} := \int \bar{Y}_{C} d\bar{\bP}=\E \int_0^1 Y_{C_s} ds  =\E \int_0^{\infty} Y_s dA_s $$ 
with $Y=\Phi(\sqrt{[\bar{X}]})$ and $Y=\Phi(\bar{X}^*)$ to obtain \eqref{BDGdA}.

\section{Change of Measure}
\label{sec-Change of Measure}
As we mentioned, the proof of the equivalence $M\in BMO(\P)\Leftrightarrow $ \eqref{BDGP} is not simple; in this section we show how to easily get the following weaker\footnote{In this statement the role of $\bP$ and $\hat{\bP}$ is reversed, and  $\hat{\bP}$ is chosen so that $d\hat{W}:=dW + \mu dt$ is a $\hat{\bP}$-Brownian motion. Given $\hat{M}:=\mu \cdot \hat{W}$, the local $\hat{\bP}$-martingale $exp(\hat{M}-[\hat{M}]/2)$ is in $BMO(\hat{\bP})$ since it has bounded quadratic variation, so this theorem is a special case of the result above.} statement.
\begin{theorem}
\label{ }
For every  $s,T\geq 0$ there exist $c,C$ such that 
\begin{align}
\label{BDGX}
   	  c \E     \sqrt{[X]_{\tau}}\leq \E X_{\tau}^* \leq C  \E    \sqrt{[X]_{\tau}} 
\end{align} 
holds for all  stopping times  $\tau$ and $X$ satisfying\footnote{We are not assuming that this SDE has solution with explosion time strictly bigger than $\tau$: this will depend on $\sigma$. We are saying that, \emph{if} $(X_t)_{t\in [0,\infty)}$ is a such a solution, then \eqref{BDGX} holds.} $dX= \sigma X (dW + \mu dt)$ on $[0,\tau]$ for some  predictable  $\sigma,\mu$ such that $|\mu|$ and $|\sigma \mu|$ are bounded by $s$ and  $\mu=0$ on $(T,\infty)$.
\end{theorem} 

\begin{proof}
Since $X$ satisfies $dX= \sigma X (dW + \mu dt)$, $X$ is continuous and $\int_0^{t\wedge \tau} \sigma^2 X^2  ds<\infty$ a.s. for all $t<\infty$ (otherwise $\int_0^{t} \sigma X dW $ is not defined on $[0,\tau]$).
 Thus, if $H_t:= X_t/ \sqrt{[X]_t + (X_t^*)^2}  \in [-1,1]$, $M_t:=\int_0^{t\wedge \tau} H\sigma X dW$ is well defined and a local martingale, and so there exist a localizing sequence  $(\tau_n)_n$ for it and  we get 
$$ \hspace{-0.5cm} | \E(H\cdot X)_{\tau_n \wedge \tau } | = | \E \int_0^{\tau_n \wedge \tau } H  \sigma X \mu\,  dt| \leq  s \E \int_0^{T\wedge \tau} |  X | \,  dt \leq    s T \E  X_\tau^* , $$
and thus applying \eqref{eqPathDCadlag}, localizing, taking expectations and passing to the limit we conclude $\E \sqrt{[X]}_\tau   \leq (3+sT) \E X^*_\tau $.
Analogously since by Holder inequality  
$$ | \E \int_0^{\tau_n \wedge \tau} H  \sigma X \mu\,  dt| \leq   s \E \int_0^{\infty} 1_{[0,T]} (1_{[0,\tau]} | \sigma X |) \,  dt \leq
 s\E \sqrt{T}  \sqrt{  \int_0^{ \tau}  \sigma^2 X^2 \,  dt } $$
 and $ \int_0^{ \tau}  \sigma^2 X^2 \,  dt = [X]_\tau$, we can conclude that $     \E    X^*_\tau  \leq (6  + 2s\sqrt{T}) \, \E\sqrt{[X]}_\tau \, $.
\end{proof} 

The reader may be interested in knowing that there are several   conditions equivalent to  \eqref{BDGP}  which one can impose on $Z$: the reverse Holder inequality, the Muckenhoupt $(A_p)$ condition and yet another unnamed $(B_p)$ condition (see  Theorem 3.4,  Corollary 3.4 and  Theorem\footnote{This Theorem has a typo, it should be $p>1$ not $p\geq 1$.} 2.4 in \cite{kaz94}).
Moreover, in \cite{BoLe79} one can find (under additional conditions) an extension of the implication  $M\in BMO(\P) \Rightarrow$ \eqref{BDGP} in the case where not every martingale is continuous.

\section{Higher Dimension}
\label{sec-HigherD}
The pathwise BDG inequalities \eqref{eqPathDCadlag} and  \eqref{PathBGCadlag}, so far stated and proved in dimension $1$, automatically hold in any finite dimension (with worse constants).
Indeed, since on $\bR^n$ all norms are equivalent,  if $|| \cdot ||_{\bR^n}$ denotes the Euclidian norm there exist  $0<\alpha_n< \beta_n$ such that   for every $n$-dimensional semimartingale $X=(X^i)_{i=1,\ldots, n}$
$$\alpha_n  \sum_{i=1}^n  (X^i)_t^* = 
\alpha_n  \sum_{i=1}^n  \sup_{s\leq t} |X^i_s| 
\leq X^*_t:=\sup_{s\leq t} ||X||_{\bR^n} \leq 
\beta_n  \sum_{i=1}^n  (X^i)_t^*  . $$
Thus the fact that Davis inequalities \eqref{eqPathDCadlag} hold for $X$ 
follows summing over $i$ the corresponding inequalities for $X^i$
(since $[X]=\sum_i [X^i]$ and $(H \cdot X):=\sum_i (H^i \cdot X^i)$).
Similarly one obtains \eqref{PathBGCadlag}, using also the fact that for all $t_i\geq 0$
$$ \frac{1}{n}\sum_{i=1}^n \Phi(t_i) \leq  \Phi(\sum_{i=1}^n t_i)
\leq n^{p-1} \sum_{i=1}^n \Phi(t_i) \, ;$$
these inequalities hold since $\Phi$ is increasing, convex and satisfies  $ \Phi(nt)\leq n^p \Phi(t)$.

Unfortunately, all this falls short of what one can do with the classic BDG inequalities, for which one can not only 
apply the above reasoning, but also prove that automatically they hold for every martingale $M$ with values in a Hilbert space $H$, and with the same constant as for $\bR^2$. In fact, one can easily construct (possibly on an enlarged probability space) a $\bR^2$-valued  martingale $N$ such that 
 $|| M_t ||_{H} =|| N_t ||_{\bR^2}$ and $ [M]_{t} = [N]_{t} $ for all $t\geq 0$. For the simple proof of this nice yet not so well-known result of \cite{KaSz91} in the discrete time case  see \cite[proposition 5.8.3]{KwaWoy92}. For the general (much harder) cadlag case one can consult  \cite{KaSz91};  notice however that one does \emph{not} need this to prove the BDG inequalities for cadlag martingales, as these follow from their discrete time version!

\end{document}